\documentclass[11pt]{article}
\usepackage{latexsym,amssymb,amsmath}
\usepackage{amssymb,amsthm,upref,amscd}
\usepackage{color}
\usepackage[T1]{fontenc}
\begin{document}
\baselineskip=15pt

\numberwithin{equation}{section}

\newtheorem{thm}{Theorem}[section]
\newtheorem{lem}[thm]{Lemma}
\newtheorem{cor}[thm]{Corollary}
\newtheorem{Prop}[thm]{Proposition}
\newtheorem{Def}[thm]{Definition}
\newtheorem{Rem}[thm]{Remark}
\newtheorem{Ex}[thm]{Example}

\newcommand{\A}{\mathbb{A}}
\newcommand{\B}{\mathbb{B}}
\newcommand{\C}{\mathbb{C}}
\newcommand{\D}{\mathbb{D}}
\newcommand{\E}{\mathbb{E}}
\newcommand{\F}{\mathbb{F}}
\newcommand{\G}{\mathbb{G}}
\newcommand{\I}{\mathbb{I}}
\newcommand{\J}{\mathbb{J}}
\newcommand{\K}{\mathbb{K}}
\newcommand{\M}{\mathbb{M}}
\newcommand{\N}{\mathbb{N}}
\newcommand{\Q}{\mathbb{Q}}
\newcommand{\R}{\mathbb{R}}
\newcommand{\T}{\mathbb{T}}
\newcommand{\U}{\mathbb{U}}
\newcommand{\V}{\mathbb{V}}
\newcommand{\W}{\mathbb{W}}
\newcommand{\X}{\mathbb{X}}
\newcommand{\Y}{\mathbb{Y}}
\newcommand{\Z}{\mathbb{Z}}
\newcommand\ca{\mathcal{A}}
\newcommand\cb{\mathcal{B}}
\newcommand\cc{\mathcal{C}}
\newcommand\cd{\mathcal{D}}
\newcommand\ce{\mathcal{E}}
\newcommand\cf{\mathcal{F}}
\newcommand\cg{\mathcal{G}}
\newcommand\ch{\mathcal{H}}
\newcommand\ci{\mathcal{I}}
\newcommand\cj{\mathcal{J}}
\newcommand\ck{\mathcal{K}}
\newcommand\cl{\mathcal{L}}
\newcommand\cm{\mathcal{M}}
\newcommand\cn{\mathcal{N}}
\newcommand\co{\mathcal{O}}
\newcommand\cp{\mathcal{P}}
\newcommand\cq{\mathcal{Q}}
\newcommand\rr{\mathcal{R}}
\newcommand\cs{\mathcal{S}}
\newcommand\ct{\mathcal{T}}
\newcommand\cu{\mathcal{U}}
\newcommand\cv{\mathcal{V}}
\newcommand\cw{\mathcal{W}}
\newcommand\cx{\mathcal{X}}
\newcommand\ocd{\overline{\cd}}

\def\c{\centerline}
\def\ov{\overline}
\def\emp {\emptyset}
\def\pa {\partial}
\def\bl{\setminus}
\def\op{\oplus}
\def\sbt{\subset}
\def\un{\underline}
\def\al {\alpha}
\def\bt {\beta}
\def\de {\delta}
\def\Ga {\Gamma}
\def\ga {\gamma}
\def\lm {\lambda}
\def\Lam {\Lambda}
\def\om {\omega}
\def\Om {\Omega}
\def\sa {\sigma}
\def\vr {\varepsilon}
\def\va {\varphi}

\title{\bf Groundstates for nonlinear fractional Choquard equations with general nonlinearities\thanks{Partially supported by NSFC (11101374, 11271331)}}

\author{Zifei Shen,\ \ Fashun Gao,\ \ Minbo Yang\thanks{M. Yang is the corresponding author: mbyang@zjnu.edu.cn}
\\
\\
{\small Department of Mathematics, Zhejiang Normal University} \\ {\small  Jinhua, Zhejiang, 321004, P. R. China}}

\date{}
\maketitle

\begin{abstract}
We study the following nonlinear Choquard equation driven by fractional Laplacian:
$$
(-\Delta)^{s}u+ u
=\big(|x|^{-\mu}\ast F(u)\big)f(u)\hspace{4.14mm}\mbox{in}\hspace{1.14mm} \mathbb{R}^N,
$$
where $N\geq3$, $s\in(0,1)$ and $\mu\in(0,N)$. By supposing that the nonlinearities satisfy the general Berestycki-Lions type conditions \cite{BL}, we are able to prove the existence of groundstates for this equation by variational methods.
 \vspace{0.3cm}


\vspace{0.3cm}

\end{abstract}

\section{Introduction and main results}

\ \ \ \ The aim of this paper is to consider the following nonlinear Choquard equation involving a fractional Laplacian:

\begin{equation}\label{CFL}
\left\{\begin{array}{l}
\displaystyle(-\Delta)^{s}u+ u
=\big(|x|^{-\mu}\ast F(u)\big)f(u)\hspace{4.14mm}\mbox{in}\hspace{1.14mm} \mathbb{R}^N,\\
\displaystyle u\in H^{s}(\mathbb{R}^N),\hspace{10.6mm}
\end{array}
\right.
\end{equation}
where $N\geq3$, $s\in(0,1)$, $\mu\in(0,N)$ and $F(u)=\int_{0}^{u}f(\tau)d\tau\in C^{1}(\mathbb{R},\mathbb{R})$. The
fractional Laplacian $(-\Delta)^{s}$ of a function $v:\mathbb{R}^N\rightarrow\mathbb{R}$ is defined by
\begin{equation}\label{FLF}
\mathcal{F}((-\Delta)^{s}v)(\xi)=|\xi|^{2s}\mathcal{F}(v)(\xi),
\end{equation}
where $\mathcal{F}$ is the Fourier transform. The operator $(-\Delta)^{s}$ can be seen as the infinitesimal generators of L\'{e}vy stable diffusion processes (see \cite{AP, La}).  The L\'{e}vy processes, extending Brownian
walk models in a natural way, occur widely in physics, chemistry and biology
and recently the stable L\'{e}vy processes, that give rise to equations with
the fractional Laplacians, have attracted much attention from many mathematicians. When $v$ is
smooth enough, it can also be computed by the following singular integral:
\begin{equation}\label{FLPV}
(-\Delta)^{s}v(x)=c_{N,s}P.V.\int_{\mathbb{R}^N}\frac{v(x)-v(y)}{|x-y|^{N+2s}}dy,
\end{equation}
where $P.V.$ is used abbreviation for "in the sense of principal value" and $c_{N,s}$ is a normalization constant. This integral makes sense directly when $s<\frac12$ and $v\in C^{0,\alpha}(\R^N)$ with $\alpha>2s$, or  $v\in C^{1,\alpha}(\R^N)$ with $1+\alpha>2s$. Here we say that a function $u\in C(\R^N)$ is a classical
solution of of \eqref{CFL} if $(-\Delta)^{s}$ can be written as \eqref{FLPV} and the equation is satisfied pointwise in all $\R^N$. Obviously, $(-\Delta)^{s}$ on $\R^N$ with $s\in (0,1)$ is a nonlocal operator. In the remarkable paper by Caffarelli and Silvestre \cite{CS2}, the authors introduced the $s-$harmonic extension technique and allow to transform the nonlocal problem into a local one via the Dirichlet-to-Neumann map. For $u\in H^{s}(\mathbb{R}^N)$, one calls its $s-$harmonic
extension $w\doteq E_{s}(u)$ by the solution to the following problem
$$
\left\{\begin{array}{l}
\displaystyle -div(y^{1-2s}\nabla w)=0\hspace{4.14mm} \mbox{in}\hspace{1.14mm} \mathbb{R}_{+}^{N+1},\\
\displaystyle \hspace{13.14mm}w(x,0)=u\hspace{4.14mm} \mbox{on}\hspace{1.14mm} \mathbb{R}^{N}.
\end{array}
\right.
$$
As shown in \cite{CS}, $(-\Delta)^{s}$ can also be characterized by
\begin{equation}\label{FLDN}
(-\Delta)^{s}u(x)=-\frac{1}{\kappa_{s}}\lim_{y\rightarrow0^{+}}y^{1-2s}\frac{\partial w}{\partial y}(x,y),\hspace{8.14mm}\forall u\in H^{s}(\mathbb{R}^N).
\end{equation}
Thus we can see that the problem \eqref{CFL} can be transformed into the following local problem
\begin{equation}\label{CFLL}
\left\{\begin{array}{l}
\displaystyle -div(y^{1-2s}\nabla w)=0\hspace{30.64mm} \mbox{in}\hspace{1.14mm} \mathbb{R}_{+}^{N+1},\\
\displaystyle  \partial_{\nu}^{s}w= -u+\big(|x|^{-\mu}\ast F(u)\big)f(u)  \hspace{12.6mm} \mbox{on}\hspace{1.14mm} \mathbb{R}^{N},
\end{array}
\right.
\end{equation}
where
$$
\partial_{\nu}^{s}w(x,0)\doteq-\frac{1}{\kappa_{s}}\lim_{y\rightarrow0^{+}}y^{1-2s}\frac{\partial w}{\partial y}(x,y), \ \ \forall\ \  x\in\mathbb{R}^N.
$$

When $s=1$, \eqref{CFL} becomes the generalized Choquard equation:
\begin{equation}\label{Nonlocal.S1}
-\Delta u+ u=\big(|x|^{-\mu}\ast F(u)\big)f(u)\ \mbox{in}\ \mathbb{R}^N.
\end{equation}
In recent years, the problem of existence and properties of the solutions for the nonlinear Choquard equation \eqref{CFL} have attracted a lot of attention. In \cite{L1}, Lieb  proved the existence and uniqueness, up to translations,
of the ground state to equation \eqref{Nonlocal.S1}. In \cite{Ls}, Lions showed the existence of
a sequence of radially symmetric solutions. Involving the properties of the ground state solutions,
 Ma and Zhao \cite{MZ} considered the generalized Choquard equation \eqref{Nonlocal.S1} for $q\geq 2$, and they proved that every positive solution of (\ref{Nonlocal.S1}) is radially symmetric and monotone decreasing about some point, under the assumption that a certain
set of real numbers, defined in terms of $N, \alpha$ and $q$, is nonempty. Under the
same assumption, Cingolani, Clapp and Secchi \cite{CCS1}  gave some existence and
multiplicity results in the electromagnetic case, and established the regularity and
some decay asymptotically at infinity of the ground states of \eqref{Nonlocal.S1}. In \cite{MS1}, Moroz and Van
Schaftingen  eliminated this restriction and showed the regularity, positivity
and radial symmetry of the ground states for the optimal range of parameters, and
derived decay asymptotically at infinity for them as well. Moreover, Moroz and Van
Schaftingen in \cite{MS2} also considered  the existence of ground states under the assumption of Berestycki-Lions type.

For fractional Laplacian with local type nonlinearities, Frank and Lenzmann \cite{FL} proved uniqueness of ground state solutions $Q=Q(|x|)>0$ for the nonlinear equation $(-\Delta)^{s}Q+Q-Q^{\alpha+1}=0$ in $\mathbb{R}$. In \cite{FQT}, Felmer, Quaas and Tan proved the existence of positive solutions of nonlinear Schr\"{o}dinger equation with fractional Laplacian in $\mathbb{R}^N$. For more investigations on fractional Laplacian, one can see \cite{BCDS2, BCDS1, CS2, CS, CW, RS} and references therein. For fractional Laplacian with nolocal Hartree type nonlinearities, the problem has also attractted a lot of interest. in the case $s=\frac{1}{2}$, Frank and Lenzmann \cite{FL1} proved analyticity and radial symmetry of ground state solutions $u>0$ for the $L^{2}$-critical boson star equation
$$
\sqrt{-\Delta}u-(|x|^{-1}\ast |u|^{2})u=-u\hspace{4.14mm}\mbox{in}\hspace{1.14mm} \mathbb{R}^N.
$$
Coti Zelati and Nolasco \cite{CN1, CN2} obtained the existence of a ground state of some fractional Schr\"{o}dinger equation with the operator $(\sqrt{-\Delta+m^{2}})$. Lei \cite{Le} considered positive solutions of a fraction order equation:
$$
(I-\Delta)^{s}u=pu^{p-1}(|x|^{-\mu}\ast|u|^{p}) \hspace{3.14mm}\mbox{in}\hspace{1.14mm} \mathbb{R}^N,
$$
with the suitable assumption of $N,s, p,\mu$.
In \cite{DSS}, D'avenia, Siciliano and Squassina obtained regularity, existence, nonexistence, symmetry as well as decays properties of ground state solutions for the nonlocal problem
$$
(-\Delta)^{s}u+ \omega u
=(|x|^{-\mu}\ast |u|^{p})|u|^{p-2}u \hspace{4.14mm}\mbox{in}\hspace{1.14mm} \mathbb{R}^N.
$$

In this paper we are going to prove the existence of solutions to Choquard equation \eqref{CFL}, we assume that nonlinearity $f\in C(\mathbb{R},\mathbb{R})$ satisfies the general Berestycki-Lions type assumption \cite{BL}:

$(f_1)$ There exists $C>0$ such that for every $t\in\mathbb{R}$,
$$
|tf(t)|\leq C(|t|^{2}+|t|^{\frac{2N-\mu}{N-2s}});
$$

$(f_2)$ Let $F:t\in\mathbb{R}\mapsto\int_{0}^{t}f(\tau)d\tau$ and suppose
$$
\lim_{t\rightarrow0}\frac{F(t)}{|t|^{2}}=0\hspace{2.14mm}\mbox{and}\hspace{2.14mm}\lim_{t\rightarrow\infty}\frac{F(t)}{|t|^{\frac{2N-\mu}{N-2s}}}=0;
$$

$(f_3)$ There exists $t_{0}\in\mathbb{R}$ such that $F(t_{0})\neq0$.

 In the present paper, $H^{s}(\mathbb{R}^N)$ is the usual fractional Sobolev space defined by
$$
H^{s}(\mathbb{R}^N)=\{u\in L^{2}(\mathbb{R}^N):\int_{\mathbb{R}^{N}}(|\xi|^{2s}\hat{u}^{2}+\hat{u}^{2})d\xi<\infty\}
$$
with equipped norms
$$
\|u\|_{H}=(\int_{\mathbb{R}^{N}}(|\xi|^{2s}\hat{u}^{2}+\hat{u}^{2})d\xi)^{\frac{1}{2}},
$$
where $\hat{u}\doteq\mathcal{F}(u)$. Notice that, for every $u\in H^{s}(\mathbb{R}^N)$, there holds
$$
\int_{\mathbb{R}^{N}}|(-\Delta)^{\frac{s}{2}}u(x)|^{2}=
\int_{\mathbb{R}^{N}}(\widehat{(-\Delta)^{\frac{s}{2}}u(\xi)})^{2}d\xi=
\int_{\mathbb{R}^{N}}(|\xi|^{s}\hat{u}(\xi))^{2}d\xi=
\int_{\mathbb{R}^{N}}|\xi|^{2s}\hat{u}^{2}d\xi<\infty,
$$
and it follows from Plancherel's theorem that
$$
\|u\|_{H}=(\int_{\mathbb{R}^{N}}(|(-\Delta)^{\frac{s}{2}}u(x)|^{2}+u^{2}))^{\frac{1}{2}}.
$$
From Lemma 2.1 of \cite{FQT}, we know $H^{s}(\mathbb{R}^N)$ continuously embedded into $L^p(\R^N)$ for $p\in [2, 2^{\ast}(s)]$ and compactly embedded into $L^p_{loc}(\R^N)$ for $p\in [2, 2^{\ast}(s))$, where $2^{\ast}(s)=2N/(N-2s)$.

Since we are concerned with the nonlocal problems, we would like to recall the well-known Hardy-Littlewood-Sobolev inequality.
\begin{Prop}\label{HLS}
 (Hardy-Littlewood-Sobolev inequality). (See \cite{LL}.) Let $t,r>1$ and $0<\mu<N$ with $1/t+\mu/N+1/r=2$, $f\in L^{t}(\mathbb{R}^N)$ and $h\in L^{r}(\mathbb{R}^N)$. There exists a sharp constant $C(t,N,\mu,r)$, independent of $f,h$, such that
$$
\int_{\mathbb{R}^{N}}\int_{\mathbb{R}^{N}}\frac{f(x)h(y)}{|x-y|^{\mu}}\leq C(t,N,\mu,r) |f|_{t}|h|_{r}.
$$
\end{Prop}

\begin{Rem}
In general, if $F(t)=|t|^{q_{0}}$ for some $q_{0}>0$. By Har-Littlewood-Sobolev inequality,
$$
\int_{\mathbb{R}^{N}}\int_{\mathbb{R}^{N}}\frac{F(u(x))F(u(y))}{|x-y|^{\mu}}
$$
is well defined if $F(u)\in L^{t}(\mathbb{R}^N)$ for $t>1$ defined by
$$
\frac{2}{t}+\frac{\mu}{N}=2.
$$
Thus, for $u\in H^{s}(\mathbb{R}^N)$, there must hold
$$
\frac{2N-\mu}{N}\leq q_{0}\leq\frac{2N-\mu}{N-2s}.
$$
Moreover, if $f$ satisfies $(f_1)$ with $0<\mu<2s$, then we know
$$
2<\frac{2N-\mu}{N-2s}.
$$
\end{Rem}
Associated to equation \eqref{CFL}, we can introduce the energy functional given by
$$
J(u)=\frac{1}{2}\int_{\mathbb{R}^N}|(-\Delta)^{\frac{s}{2}}u|^{2}+\frac{1}{2}\int_{\mathbb{R}^N} |u|^{2}-\frac{1}{2}\int_{\mathbb{R}^N}\big(|x|^{-\mu}\ast F(u)\big)F(u),
$$
clearly, $J$ is well defined on $H^{s}(\mathbb{R}^N)$ and belongs to $C^1$. We say a function $u$ is a weak solution of \eqref{CFL}, if $u\in H^{s}(\mathbb{R}^N)\backslash\{{0}\}$ and satisfies
$$
\int_{\mathbb{R}^N}(-\Delta)^{\frac{s}{2}}u(-\Delta)^{\frac{s}{2}}\varphi  +\int_{\mathbb{R}^N}u\varphi  =\int_{\mathbb{R}^N}\big(|x|^{-\mu}\ast F(u)\big)f(u)\varphi
$$
for all $\varphi\in C_{0}^{\infty}(\mathbb{R}^N)$.  And so  $u$ is a weak solution of \eqref{CFL} if and only if $u$ is a critical point of functional $J$.  Furthermore, a function $u_{0}$ is called a ground state of \eqref{CFL} if $u_{0}$ is a critical point of \eqref{CFL} and $J(u_{0})\leq J(u)$ holds for any critical point $u$ of \eqref{CFL}, i.e.
\begin{equation}\label{c}
J(u_{0})=c:=\inf \Big\{J(u):u\in H^{s}(\mathbb{R}^N)\backslash\{{0}\} \mbox{ is a critical point of \eqref{CFL}} \Big\}.
\end{equation}

The main result of this paper is to establish the existence of groundstate solution for \eqref{CFL} under assumptions $(f_1)--(f_3)$. The result says that
\begin{thm}\label{EXS}
 Assume that $N\geq3$, $s\in(0,1)$ and $\mu\in(0,2s)$. If $f\in C^1(\mathbb{R},\mathbb{R})$ satisfies $(f_1)$, $(f_2)$ and $(f_3)$, then equation \eqref{CFL} has at least one ground state solution.
\end{thm}
In order to prove Theorem \ref{EXS}, we will apply the mountain pass Theorem \cite{MW}. Define the mountain pass level
\begin{equation}\label{c2}
c^{\star}=\inf\limits_{\gamma\in\Gamma}\sup\limits_{t\in [0,1]}J(\gamma(t)),
\end{equation}
with the set of paths defined by
\begin{equation}\label{gamma}
\Gamma=\Big\{\gamma\in C([0,1];H^{s}(\mathbb{R}^N)):\gamma(0)=0,J(\gamma(1))<0\Big\}.
\end{equation}
We will show that $c^{\star}$ is a critical value of the functional $J$ and there exists $u_{0}\in H^{s}(\mathbb{R}^N)$ such that $J(u_{0})=c^{\star}$. In order to prove $u_{0}$ is a ground state of \eqref{CFL}, we need to prove $c=c^{\star}$. The relationship between critical levels $c^{\star}$ and $c$ is established via the construction of paths associated to critical points by following the methods of Jeanjean and Tanaka in \cite{JT}.

The paper is organized as follows: In section 2, we obtain the existence of nontrivial solution by Mountain-Pass Theorem and Concentration-Compactness arguments. In section 3, we prove some regularity results for the weak solutions and establish the Poho\v{z}aev identity. In section 4, we show the existence of groundstates for \eqref{CFL} and study the symmetric property of these groundstates.

\section{Weak Solutions}

\ \ \ \ Throughout this paper we write $|\cdot|_{q}$ for the $L^{q}(\mathbb{R}^N)$-norm for $q\in[1,\infty]$ and $\|\cdot\|_{q}$ for the $L^{q}(\mathbb{R}_{+}^{N+1})$-norm for $q\in[1,\infty]$.  We denote positive constants by $C, C_{1}, C_{2}, C_{3} \cdots$.

The following embedding Lemma for fractional Sobolev space $H^{s}(\mathbb{R}^N)$ is proved in \cite{FQT}.
\begin{lem}\label{EMB}
Let $2\leq q\leq2^{*}(s)=2N/(N-2s)$, then we have
$$
|u|_{q}\leq C_q\|u\|_{H},\hspace{4.14mm}\forall u\in H^{s}(\mathbb{R}^N).
$$
If further $2\leq q<2^{*}(s)$ and $\Omega\subset\mathbb{R}^N$ is a bounded domain then bound sequence $\{u_{k}\}\subset H^{s}(\mathbb{R}^N)$ has a convergent subsequence in $L^{q}(\Omega)$.
\end{lem}

\begin{lem}\label{MPG}
Assume that $f\in C(\mathbb{R},\mathbb{R})$ satisfies $(f_1)$, $(f_2)$ and $(f_3)$, then
\begin{itemize}
  \item[$(i).$]  There exist $\rho>0$ and $\beta>0$ such that $J|_S\geq\beta$ for all $u\in S=\{u\in H^{s}(\mathbb{R}^N):\|u\|_{H}=\rho\}$.
  \item[$(ii).$]  There exists $v\in H^{s}(\mathbb{R}^N)$ such that $v\geq0$ $ a.e.$ on $\mathbb{R}^N$, $\|v\|_{H}>\rho$ and $J(v)<0$, where $\rho$ is given in (i).
 \end{itemize}
\end{lem}
\begin{proof}
(i) By the growth assumption $(f_1)$, Lemma \ref{EMB} and Proposition \ref{HLS}, for every $u\in H^{s}(\mathbb{R}^N)$, there holds
$$\aligned
\int_{\mathbb{R}^{N}}&\big(|x|^{-\mu}\ast F(u)\big)F(u)\\
&\leq C_{1}(\int_{\mathbb{R}^{N}}|F(u)|^{\frac{2N}{2N-\mu}})^{\frac{2N-\mu}{N}}\\
&\leq C_{2}\big(\int_{\mathbb{R}^{N}}(|u|^{2}+|u|^{\frac{2N}{N-2s}})\big)^{\frac{2N-\mu}{N}}\\
&\leq C_{3}\Big((\int_{\mathbb{R}^{N}}|u|^{\frac{4N}{2N-\mu}})^{\frac{2N-\mu}{N}}+(\int_{\mathbb{R}^{N}}\big(|(-\Delta)^{\frac{s}{2}}u(x)|^{2}+u^{2}\big))^{\frac{2N-\mu}{N-2s}}\Big)\\
&\leq C_{4}\Big((\int_{\mathbb{R}^{N}}\big(|(-\Delta)^{\frac{s}{2}}u(x)|^{2}+u^{2}\big))^{2}+(\int_{\mathbb{R}^{N}}\big(|(-\Delta)^{\frac{s}{2}}u(x)|^{2}+u^{2}\big))^{\frac{2N-\mu}{N-2s}}\Big).\\
\endaligned
$$
Note that $\frac{2N-\mu}{N-2s}>1$, hence there exists $\delta>0$ such that if $$\int_{\mathbb{R}^{N}}(|(-\Delta)^{\frac{s}{2}}u(x)|^{2}+u^{2})\leq\delta,$$ then
$$
\int_{\mathbb{R}^{N}}\big(|x|^{-\mu}\ast F(u)\big)F(u)\leq \frac{1}{4}\int_{\mathbb{R}^{N}}(|(-\Delta)^{\frac{s}{2}}u(x)|^{2}+u^{2})
$$
and so
$$
J(u)\geq\frac{1}{4}\int_{\mathbb{R}^{N}}(|(-\Delta)^{\frac{s}{2}}u(x)|^{2}+u^{2}).
$$
So, we can choose $\rho$ sufficiently small and $\beta>0$, such that $J|_S\geq\beta$ for all $u\in S=\{u\in H^{s}(\mathbb{R}^N):\|u\|_{H}=\rho\}$.

(ii) From assumption $(f_3)$, we can choose $t_{0}\in\mathbb{R}$ such that $F(t_{0})\neq0$. Let $w=t_{0}\chi_{B_{1}}$, we get
$$
\int_{\mathbb{R}^{N}}\big(|x|^{-\mu}\ast F(w)\big)F(w)=F(t_{0})^{2} \int_{B_{1}}\int_{B_{1}}|x-y|^{-\mu}>0.
$$
Since $H^{s}(\mathbb{R}^N)$ is dense in $L^{2}(\mathbb{R}^N)\cap L^{\frac{2N}{N-2s}}(\mathbb{R}^N)$ and $\int_{\mathbb{R}^{N}}\big(|x|^{-\mu}\ast F(u)\big)F(u)$ is continuous in $L^{2}(\mathbb{R}^N)\cap L^{\frac{2N}{N-2s}}(\mathbb{R}^N)$, we know there exists $v\in H^{s}(\mathbb{R}^N)$ such that
$$
\int_{\mathbb{R}^{N}}\big(|x|^{-\mu}\ast F(v)\big)F(v)>0.
$$
Defined for $\tau>0$ and $x\in \mathbb{R}^{N}$
by $u_{\tau}(x)=v(\frac{x}{\tau})$. we find that, for $\tau>0$,
\begin{equation} \label{M1}
\aligned
J(u_{\tau})&=\frac{1}{2}\int_{\mathbb{R}^N}|(-\Delta)^{\frac{s}{2}}v(\frac{x}{\tau})|^{2}+\frac{1}{2}\int_{\mathbb{R}^N} |v(\frac{x}{\tau})|^{2}-\frac{1}{2}\int_{\mathbb{R}^N}\big(|x|^{-\mu}\ast F(v(\frac{x}{\tau}))\big)F(v(\frac{x}{\tau}))\\
&=\frac{1}{2}\int_{\mathbb{R}^N}\big(c_{N,s}P.V.\int_{\mathbb{R}^N}\frac{v(\frac{x}{\tau})-v(\frac{y}{\tau})}{|x-y|^{N+s}}\big)^{2}+\frac{\tau^{N}}{2}\int_{\mathbb{R}^N} |v|^{2}-\frac{\tau^{2N-\mu}}{2}\int_{\mathbb{R}^N}\big(|x|^{-\mu}\ast F(v)\big)F(v)\\
&=\frac{\tau^{N-2s}}{2}\int_{\mathbb{R}^N}\big(c_{N,s}P.V.\int_{\mathbb{R}^N}\frac{v(x)-v(y)}{|x-y|^{N+s}}\big)^{2} +\frac{\tau^{N}}{2}\int_{\mathbb{R}^N} |v|^{2}-\frac{\tau^{2N-\mu}}{2}\int_{\mathbb{R}^N}\big(|x|^{-\mu}\ast F(v)\big)F(v)\\
&=\frac{\tau^{N-2s}}{2}\int_{\mathbb{R}^N}|(-\Delta)^{\frac{s}{2}}v|^{2}+\frac{\tau^{N}}{2}\int_{\mathbb{R}^N} |v|^{2}-\frac{\tau^{2N-\mu}}{2}\int_{\mathbb{R}^N}\big(|x|^{-\mu}\ast F(v)\big)F(v).
\endaligned
\end{equation}
Observe that, for $\tau>0$ large enough, $J(u_{\tau})<0$. By the proof of (i), we also know $\|u_{\tau}\|_{H}>\rho$. So, the assertion follows by taking $v=u_{\tau}$, with $\tau$ sufficiently large.
\end{proof}

From the Min-Max characterization of the value $c^{\star}$, we can see $0<c^{\star}<\infty$. It is convenient to define Poho\v{z}aev functional $P:H^{s}(\mathbb{R}^N)\rightarrow\mathbb{R}$ for $u\in H^{s}(\mathbb{R}^N)$ by
\begin{equation}\label{POH}
P(u)=\frac{N-2s}{2}\int_{\mathbb{R}^N}|(-\Delta)^{\frac{s}{2}}u|^{2}+\frac{N}{2}\int_{\mathbb{R}^N} |u|^{2}-\frac{2N-\mu}{2}\int_{\mathbb{R}^N}\big(|x|^{-\mu}\ast F(u)\big)F(u).
\end{equation}
In order to construct a Poho\v{z}aev-Palais-Smale sequence, following Jeanjean \cite{Je},  for $\sigma\in\mathbb{R}$, $v\in H^{s}(\mathbb{R}^N)$ and $x\in\mathbb{R}$, we define the map $\Phi:\mathbb{R}\times H^{s}(\mathbb{R}^N)\rightarrow H^{s}(\mathbb{R}^N)$ by
$$
\Phi(\sigma,v)(x)=v(e^{-\sigma}x).
$$
Then, the functional $J\circ \Phi$ is computed as
$$
J(\Phi(\sigma,v))=\frac{e^{(N-2s)\sigma}}{2}\int_{\mathbb{R}^N}|(-\Delta)^{\frac{s}{2}}v|^{2}+\frac{e^{N\sigma}}{2}\int_{\mathbb{R}^N} |v|^{2}-\frac{e^{(2N-\mu)\sigma}}{2}\int_{\mathbb{R}^N}\big(|x|^{-\mu}\ast F(v)\big)F(v).
$$
Define the family
of paths
$$
\tilde{\Gamma}=\{\tilde{\gamma}\in C\big([0,1];\mathbb{R}\times H^{s}(\mathbb{R}^N)\big):\tilde{\gamma}(0)=(0,0) \hspace{1.14mm}\mbox{and}\hspace{1.14mm} J\circ \Phi(\tilde{\gamma}(1))<0\}
$$
and notice that $\Gamma=\{\Phi\circ\tilde{\gamma}:\tilde{\gamma}\in \tilde{\Gamma}\}$, we see the mountain pass levels of $J$ and $J\circ \Phi$ coincide:
$$
c^{\star}=\inf\limits_{\tilde{\gamma}\in\tilde{\Gamma}}\sup\limits_{t\in [0,1]}(J\circ \Phi)(\tilde{\gamma}(t)).
$$
Using condition $(f_1)$,  we know that $J\circ \Phi$ is continuous and Fr\'{e}chet-differentiable on $\mathbb{R}\times H^{s}(\mathbb{R}^N)$. Applying Theorem 2.9 in \cite{MW} and Lemma 2.2, there exists a sequence $((\sigma_{n},v_{n}))_{n\in\mathbb{N}}$ in $\mathbb{R}\times H^{s}(\mathbb{R}^N)$
such that as $n\rightarrow\infty$
$$
(J\circ \Phi)(\sigma_{n},v_{n})\rightarrow c^{\star},\hspace{39.14mm}
$$
$$
(J\circ \Phi)'(\sigma_{n},v_{n})\rightarrow0\hspace{10.14mm}\mbox{in}(\mathbb{R}\times H^{s}(\mathbb{R}^N))^{\ast}.
$$
Since for every $(h,w)\in\mathbb{R}\times H^{s}(\mathbb{R}^N)$,
$$
(J\circ \Phi)'(\sigma_{n},v_{n})[h,w]=J'(\Phi(\sigma_{n},v_{n}))[\Phi(\sigma_{n},w)]+P(\Phi(\sigma_{n},v_{n}))h.
$$
We take $u_{n}=\Phi(\sigma_{n},v_{n})$, then as $n\rightarrow\infty$,
\begin{equation}\label{POHS}
\aligned
J(u_{n})&\rightarrow c^{\star}>0,\\
J'(u_{n})&\rightarrow0,\\
P(u_{n})&\rightarrow0.
\endaligned
\end{equation}

Now we are ready to obtain a nontrivial solution from this special sequence by applying a version of Lions' concentration-compactness Lemma for Fractional Laplacian, see \cite{DSS}.
\begin{lem}\label{CCL}
Let $(u_{n})_{n\in\mathbb{N}}$ be a bounded sequence in $H^{s}(\mathbb{R}^N)$. For some $\sigma>0$ and $2\leq q<2^{*}(s)$ there holds
$$
\sup_{x\in\R^N}\int_{x\in B_{\sigma}(x)}|u_n|^q\to0,\  \hbox{as}\ \   n\to\infty,
$$
then $u_n\to0$ in $L^{s}(\mathbb{R}^N)$ for $2<r<2^{*}(s)$.
\end{lem}
\begin{lem}\label{Existence}
Suppose that $(f_1)-(f_3)$, then equation \eqref{CFL} has at least one nontrivial solution.
\end{lem}
\begin{proof}
 Let $(u_{n})_{n\in\mathbb{N}}$ be the sequence obtained in \ref{POHS}, then  it is bounded in $H^{s}(\mathbb{R}^N)$. In fact, for any $n\in\mathbb{N}$,
$$
J(u_{n})-\frac{1}{2N-\mu}P(u_{n})=\frac{N-\mu+2s}{2(2N-\mu)}\int_{\mathbb{R}^N}|(-\Delta)^{\frac{s}{2}}u_{n}|^{2}+\frac{N-\mu}{2(2N-\mu)}\int_{\mathbb{R}^N} |u_{n}|^{2},
$$
where $P$ is the Poho\v{z}aev function defined in \ref{POH}. Thus it is easy to see that the sequence $(u_{j})_{n\in\mathbb{N}}$ is bounded in $H^{s}(\mathbb{R}^N)$. Moreover, we claim that there exist $\sigma, \delta>0$ and a sequence $(y_n) \subset \R^N$ such that
$$
\liminf_{n\to\infty}\int_{B_\sigma(y_n)}|u_n|^2\geq\delta.
$$
If the above claim does not hold for $(u_n)$, by Lemma \ref{CCL}, we must have that
$$
u_n\to 0 \,\,\ \mbox{in} \,\,\, L^r(\R^N) \,\,\, \mbox{for} \,\,\, 2<r< 2^{*}(s).
$$
Fix $2<q<\frac{2N-\mu}{N-2s}$, from assumption $(f_2) $, for any $\xi>0$ there is $C_\xi>0$ such that
$$
|F(s)|\leq \xi(|s|^{2}+|s|^{\frac{2N-\mu}{N-2s}})+C_\xi|s|^q \quad \forall s \geq 0,
$$
it follows from Hardy-Littlewood-Sobolev inequality
$$
\aligned
\int_{\mathbb{R}^{N}}\big(|x|^{-\mu}\ast F(u_n)\big)f(u_n)u_n&\leq C_{1}(\int_{\mathbb{R}^{N}}|F(u_n)|^{\frac{2N}{2N-\mu}})^{\frac{2N-\mu}{2N}}(\int_{\mathbb{R}^{N}}|f(u_n)u_n|^{\frac{2N}{2N-\mu}})^{\frac{2N-\mu}{2N}}\\
&\leq
C_{2}\xi+C_{3}(\int_{\mathbb{R}^{N}} |u_n|^{\frac{2qN}{2N-\mu}})^{\frac{2N-\mu}{2N}} .
\endaligned
$$
Consequently,
$$
\int_{\mathbb{R}^{N}}\big(|x|^{-\mu}\ast F(u_n)\big)f(u_n)u_n\to0,
$$
which leads to $\|u_n\|_{H} \to 0$. This is an absurd with \eqref{POHS} and so the claim holds. And so, up to translation, we may assume
$$
\liminf_{n\to\infty}\int_{B_\sigma(0)}|u_n|^2\geq\delta.
$$
Using Lemma \ref{EMB},  there exists $u_{0}\in H^{s}(\mathbb{R}^N)$, $u_{0}\neq0$, such
that, up to a subsequence,  $u_{n}$ converges weakly in $ H^{s}(\mathbb{R}^N)$ and $u_{n}(x)$ converges to $u_{0}(x)$ almost everywhere in $\mathbb{R}^N$.

Since the sequence $(u_{n})_{n\in\mathbb{N}}$ is bounded in $L^{2}(\mathbb{R}^N)\cap L^{\frac{2N}{N-2s}}(\mathbb{R}^N)$, using $(f_1)$, we know the sequence $\big(F( u_{n})\big)_{n\in\mathbb{N}}$ is bounded in $L^{\frac{2N}{2N-\mu}}(\mathbb{R}^N)$. Note that $F$ is contiuous, we have $\big(F( u_{n})\big)_{n\in\mathbb{N}}$ converges almost everywhere to $F( u_{0})$
in $\mathbb{R}^N$. This implies that the sequence $\big(F( u_{n})\big)_{n\in\mathbb{N}}$ converges weakly to $F(u_{0})$ in $L^{\frac{2N}{2N-\mu}}(\mathbb{R}^N)$. As
the $|x|^{-\mu}$ defines a linear continuous map from $L^{\frac{2N}{2N-\mu}}(\mathbb{R}^N)$ to $L^{\frac{2N}{\mu}}(\mathbb{R}^N)$, the sequence $\big(|x|^{-\mu}\ast(F(u_{n})\big)_{n\in\mathbb{N}}$ converges weakly to $|x|^{-\mu}\ast(F(u_{0}))$ in $L^{\frac{2N}{\mu}}(\mathbb{R}^N)$.

Applying condition $(f_1)$ and Lemma 2.1, we can get, for every $p\in[1,\frac{2N}{N+2s-\mu})$,
$$
f( u_{n})\rightarrow f( u_{0})\hspace{6.14mm}\mbox{in}\hspace{1.14mm}L_{loc}^{p}(\mathbb{R}^N).
$$
We conclude that
$$
|x|^{-\mu}\ast(F(u_{n}))f( u_{n})\rightharpoonup |x|^{-\mu}\ast(F(u_{0}))f( u_{0})\hspace{6.14mm}\mbox{ weakly in}\hspace{1.14mm}L^{p}(\mathbb{R}^N),
$$
for every $p\in[1,\frac{2N}{N+2s})$. In particular, for every $\varphi\in H^{s}(\mathbb{R}^N)$,

$$
\aligned
0=&\lim_{n\rightarrow\infty}\langle J'(u_{n}), \varphi \rangle_H\\ &=\lim_{n\rightarrow\infty}\int_{\mathbb{R}^N}(-\Delta)^{\frac{s}{2}}u_{n}(-\Delta)^{\frac{s}{2}}\varphi  +\int_{\mathbb{R}^N} u_{n}\varphi  -\int_{\mathbb{R}^N}(|x|^{-\mu}\ast F(u_{n}))f(u_{n})\varphi  \\
&\to\int_{\mathbb{R}^N}(-\Delta)^{\frac{s}{2}}u_{0}(-\Delta)^{\frac{s}{2}}\varphi  +\int_{\mathbb{R}^N} u_{0}\varphi  -\int_{\mathbb{R}^N}(|x|^{-\mu}\ast F(u_{0}))f(u_{0})\varphi \\
&=\langle J'(u_{0}), \varphi \rangle_H,
\endaligned
$$
that is, $u_{0}$ is a weak solution of equation \eqref{CFL}.

Moreover, the weak lower-semicontinuity of the norm and the Poho\v{z}aev identity implies that
$$\aligned
J(u_{0})&=J(u_{0})-\frac{P(u_{0})}{2N-\mu}\\
&=\frac{N-\mu+2s}{2(2N-\mu)}\int_{\mathbb{R}^N}|(-\Delta)^{\frac{s}{2}}u_{0}|^{2}dx+\frac{N-\mu}{2(2N-\mu)}\int_{\mathbb{R}^N} |u_{0}|^{2}dx\\
&\leq\liminf_{n\rightarrow\infty}\frac{N-\mu+2s}{2(2N-\mu)}\int_{\mathbb{R}^N}|(-\Delta)^{\frac{s}{2}}u_{n}|^{2}dx+\frac{N-\mu}{2(2N-\mu)}\int_{\mathbb{R}^N} |u_{n}|^{2}dx \\
&=\liminf_{n\rightarrow\infty} J(u_{n})-\frac{P(u_{n})}{2N-\mu}=\liminf_{n\rightarrow\infty} J(u_{n})=c^{\star}.
\endaligned
$$
\end{proof}

\section{Regularity}
 In this section we are going to show that the solutions for equation \eqref{CFL} possess some regularity which will be used to prove a Poho\v{z}aev identity  for Fractional Choquard equation. First let us recall an important inequality due to Moroz and Van Schaftingen \cite{MS2}.
\begin{lem}\label{Es}
Let $q,r,l,t\in[1,\infty)$ and $\lambda\in[0,2]$ such that
$$
1+\frac{\mu}{N}-\frac{1}{l}-\frac{1}{t}=\frac{\lambda}{q}+\frac{2-\lambda}{r},
$$
If $\theta\in(0,2)$ satisfies
$$
\mbox{max}(q,r)(\frac{\mu}{N}-\frac{1}{l})<\theta<\mbox{max}(q,r)(1-\frac{1}{l})
$$
and
$$
\mbox{max}(q,r)(\frac{\mu}{N}-\frac{1}{t})<2-\theta<\mbox{max}(q,r)(1-\frac{1}{t}),
$$
then for every $H\in L^{l}(\mathbb{R}^N), K\in L^{t}(\mathbb{R}^N)$ and $u\in L^{q}(\mathbb{R}^N)\cap L^{r}(\mathbb{R}^N)$,
$$
\int_{\mathbb{R}^N}(|x|^{-\mu}\ast (H|u|^{\theta}))K|u|^{2-\theta}\leq C(\int_{\mathbb{R}^N}|H|^{l})^{\frac{1}{l}}(\int_{\mathbb{R}^N}|K|^{t})^{\frac{1}{t}}(\int_{\mathbb{R}^N}|u|^{q})^{\frac{\lambda}{q}}(\int_{\mathbb{R}^N}|u|^{r})^{\frac{2-\lambda}{r}}.
$$
\end{lem}

In \cite{MS2}, the authors adapted the arguments of Brezis and Kato \cite{BK} and improved the integrability of solutions of a nonlocal linear elliptic equation. However, the appearance of Fractional Laplacian operator makes the proof much more complicated.

The arguments of following lemma follows the strategy of Moroz and Van Schaftingen \cite{MS2} for nonlocal linear equations dominated by the Laplacian.
\begin{lem}\label{ELT}
Let $N\geq 2$, $\mu\in (0,2s)$ and $\theta\in (0,N)$. If $H, K\in L^{\frac{2N}{N-\mu}}(\mathbb{R}^N)+L^{\frac{2N}{N+2s-\mu}}(\mathbb{R}^N) $, $(1-\frac{\mu}{N})<\theta<(1+\frac{\mu}{N})$, then for any $\vr>0$, there exists $C_{\varepsilon,\theta}\in\mathbb{R}$ such that for every $u\in H^{s}(\mathbb{R}^N)$,
$$
\int_{\mathbb{R}^N}(|x|^{-\mu}\ast (H|u|^{\theta}))K|u|^{2-\theta}\leq \varepsilon^{2}\int_{\mathbb{R}^{N}}|(-\Delta)^{\frac{s}{2}}u|^{2}+C_{\varepsilon,\theta}\int_{\mathbb{R}^N}|u|^{2}.
$$
\end{lem}

\begin{proof}

Let $(1-\frac{\mu}{N})<\theta<(1+\frac{\mu}{N})$, $u\in H^{s}(\mathbb{R}^N)$. Since $0<\mu<2s$, we may assume that $H=H^{\star}+H_{\star}$ and $K=K^{\star}+K_{\star}$ with $H^{\star},K^{\star}\in L^{\frac{2N}{N-\mu}}(\mathbb{R}^N)$ and $H_{\star},K_{\star}\in L^{\frac{2N}{N+2s-\mu}}(\mathbb{R}^N)$. We take
$$q=r=\frac{2N}{N-2s}, \ l=t=\frac{2N}{N+2s-\mu}, \ \lambda=0;$$
$$q=r=2,\ l=t=\frac{2N}{N-\mu},\ \lambda=2;$$
$$q=2,\ r=\frac{2N}{N-2s},\ l=\frac{2N}{N+2s-\mu},\ t=\frac{2N}{N-\mu}
,\lambda=1;
$$
and
$$q=2,\ r=\frac{2N}{N-2s},\ l=\frac{2N}{N-\mu},\ t=\frac{2N}{N+2s-\mu},\ \lambda=1$$ in Lemma \ref{Es} respectively. We obtain
\begin{equation} \label{E1}
\aligned
\int_{\mathbb{R}^N}\big(|x|^{-\mu}&\ast (H_{\star}|u|^{\theta})\big)K_{\star}|u|^{2-\theta}\\
&\leq C(\int_{\mathbb{R}^N}|H_{\star}|^{\frac{2N}{N+2s-\mu}})^{\frac{N+2s-\mu}{2N}}(\int_{\mathbb{R}^N}|K_{\star}|^{\frac{2N}{N+2s-\mu}})^{\frac{N+2s-\mu}{2N}}(\int_{\mathbb{R}^N}|u|^{\frac{2N}{N-2s}})^{1-\frac{2s}{N}},
\endaligned
\end{equation}
\begin{equation} \label{E2}
\aligned
\int_{\mathbb{R}^N}\big(|x|^{-\mu}\ast (H^{\star}|u|^{\theta})\big)K^{\star}|u|^{2-\theta}\leq C(\int_{\mathbb{R}^N}|H^{\star}|^{\frac{2N}{N-\mu}})^{\frac{N-\mu}{2N}}(\int_{\mathbb{R}^N}|K^{\star}|^{\frac{2N}{N-\mu}})^{\frac{N-\mu}{2N}}\int_{\mathbb{R}^N}|u|^{2},
\endaligned
\end{equation}
\begin{equation} \label{E3}
\aligned
\int_{\mathbb{R}^N}\big(|x|^{-\mu}&\ast (H_{\star}|u|^{\theta})\big)K^{\star}|u|^{2-\theta}\\
&\leq C(\int_{\mathbb{R}^N}|H_{\star}|^{\frac{2N}{N+2s-\mu}})^{\frac{N+2s-\mu}{2N}}(\int_{\mathbb{R}^N}|K^{\star}|^{\frac{2N}{N-\mu}})^{\frac{N-\mu}{2N}}(\int_{\mathbb{R}^N}|u|^{2})^{\frac{1}{2}}(\int_{\mathbb{R}^N}|u|^{\frac{2N}{N-2s}})^{\frac{1}{2}-\frac{s}{N}}
\endaligned
\end{equation}
and
\begin{equation} \label{E4}
\aligned
\int_{\mathbb{R}^N}\big(|x|^{-\mu}&\ast (H^{\star}|u|^{\theta})\big)K_{\star}|u|^{2-\theta}\\
&\leq C(\int_{\mathbb{R}^N}|H^{\star}|^{\frac{2N}{N-\mu}})^{\frac{N-\mu}{2N}}(\int_{\mathbb{R}^N}|K_{\star}|^{\frac{2N}{N+2s-\mu}})^{\frac{N+2s-\mu}{2N}}(\int_{\mathbb{R}^N}|u|^{2})^{\frac{1}{2}}(\int_{\mathbb{R}^N}|u|^{\frac{2N}{N-2s}})^{\frac{1}{2}-\frac{s}{N}}.
\endaligned
\end{equation}
Then, applying Lemma \ref{EMB} and the above inequalities, we have, for every $ u\in H^{s}(\mathbb{R}^N)$,
$$
\aligned
\int_{\mathbb{R}^N}\big(|x|^{-\mu}&\ast (H|u|^{\theta})\big)K|u|^{2-\theta}\\
&\leq C((\int_{\mathbb{R}^N}|H_{\star}|^{\frac{2N}{N+2s-\mu}}\int_{\mathbb{R}^N}|K_{\star}|^{\frac{2N}{N+2s-\mu}})^{\frac{N+2s-\mu}{2N}}\int_{\mathbb{R}^{N}}|(-\Delta)^{\frac{s}{2}}u|^{2}
\endaligned
$$
$$
\hspace{38.6mm}+(\int_{\mathbb{R}^N}|H^{\star}|^{\frac{2N}{N-\mu}}\int_{\mathbb{R}^N}|K^{\star}|^{\frac{2N}{N-\mu}})^{\frac{N-\mu}{2N}}\int_{\mathbb{R}^N}|u|^{2}).
$$
For $\varepsilon>0$, we choose $H^{\star}$ and $K^{\star}$ such that
$$
C(\int_{\mathbb{R}^N}|H_{\star}|^{\frac{2N}{N+2s-\mu}}\int_{\mathbb{R}^N}|K_{\star}|^{\frac{2N}{N+2s-\mu}})^{\frac{N+2s-\mu}{2N}}\leq\varepsilon^{2},
$$
then there exists $C_{\varepsilon,\theta}\in\mathbb{R}$ such that
$$
\int_{\mathbb{R}^N}(|x|^{-\mu}\ast (H|u|^{\theta}))K|u|^{2-\theta}\leq \varepsilon^{2}\int_{\mathbb{R}^{N}}|(-\Delta)^{\frac{s}{2}}u|^{2}+C_{\varepsilon,\theta}\int_{\mathbb{R}^N}|u|^{2}.
$$
\end{proof}

\begin{lem}\label{Integrability}
Let $N\geq 2$, $\mu\in (0,2s)$. If $H, K\in L^{\frac{2N}{N-\mu}}(\mathbb{R}^N)+L^{\frac{2N}{N+2s-\mu}}(\mathbb{R}^N) $ and $u\in H^{s}(\mathbb{R}^N)$ solves
$$
(-\Delta)^{s}u+ u
=\big(|x|^{-\mu}\ast (Hu)\big)K,
$$
then $u\in L^{p}(\mathbb{R}^N)$ for every $p\in[2,\frac{2N^{2}}{(N-\mu)(N-2s)})$. Moreover, there exists a constant $C_p$ independent of $u$ such that
$$
\big(\int_{\mathbb{R}^N}|u|^{p}\big)^{\frac1p}\leq C_p\big(\int_{\mathbb{R}^N}|u|^{2}\big)^{\frac12}.
$$
\end{lem}
\begin{proof}
Let $\theta=1$ in Lemma \ref{ELT}, there exists $\lambda>0$ such that for every $\varphi\in H^{s}(\mathbb{R}^N)$,
$$
\int_{\mathbb{R}^N}(|x|^{-\mu}\ast (|H\varphi|))|K\varphi|\leq \frac{1}{2}\int_{\mathbb{R}^{N}}|(-\Delta)^{\frac{s}{2}}\varphi|^{2}+\frac{\lambda}{2}\int_{\mathbb{R}^N}|\varphi|^{2}.
$$
We follow the strategy of \cite{BK}, take sequences $(H_{k})_{k\in\mathbb{N}}$ and $(K_{k})_{k\in\mathbb{N}}$ in $L^{\frac{2N}{N-\mu}}(\mathbb{R}^N)$ such that $|H_{k}|\leq|H|$ and $|K_{k}|\leq|K|$, and $H_{k}\rightarrow H$ and $K_{k}\rightarrow K$ almost everywhere in $\mathbb{R}^N$. For every $k\in\mathbb{N}$, the form $a_{k}:H^{s}(\mathbb{R}^N)\times H^{s}(\mathbb{R}^N)\rightarrow\mathbb{R}$ defined for $u, v\in H^{s}(\mathbb{R}^N)$ by
$$
a_{k}(u,v)=\int_{\mathbb{R}^{N}}\big((-\Delta)^{\frac{s}{2}}u(-\Delta)^{\frac{s}{2}}v+uv\big)-\int_{\mathbb{R}^N}(|x|^{-\mu}\ast (H_{k}u))K_{k}v
$$
is bilinear and coercive. Applying Lax-Milgram theorem, there exists an unique solution $u_{k}\in H^{s}(\mathbb{R}^N)$ satisfies
\begin{equation} \label{ELP}
(-\Delta)^{s}u_{k}+ \lambda u_{k}
=(|x|^{-\mu}\ast (H_{k}u_{k}))K_{k} +(\lambda-1)u
\end{equation}
and moreover the sequence $(u_{k})_{k\in\mathbb{N}}$ converges weakly to $u$ in $H^{s}(\mathbb{R}^N)$ as $k\rightarrow\infty$.

In order to continue the proof, we need the weighted function space $X^{s}(\mathbb{R}_{+}^{N+1})$ as the completion of $C_{0}^{\infty}(\mathbb{R}_{+}^{N+1})$ under the norms
\begin{equation} \label{Norm3}
\|w\|_{X^{s}(\mathbb{R}_{+}^{N+1})}=(\kappa_{s}\int_{\mathbb{R}_{+}^{N+1}}y^{1-2s}|\nabla w|^{2})^{\frac{1}{2}},
\end{equation}
where $\kappa_{s}=(2^{1-2s}\Gamma(1-s)/\Gamma(s))$.
By lemma A.2 in \cite{BCDS1}, it follows that
\begin{equation} \label{Norm4}
\|w\|_{X^{s}(\mathbb{R}_{+}^{N+1})}=(\int_{\mathbb{R}^N}|(-\Delta)^{\frac{s}{2}}u|^{2} )^{\frac{1}{2}},
\end{equation}
where $w= E_{s}(u)$. By \eqref{CFLL}, we can see that the problem \eqref{ELP} can be transformed into the following local problem
\begin{equation} \label{EQP}
\left\{\begin{array}{l}
\displaystyle -div(y^{1-2s}\nabla w_{k})=0\hspace{51.64mm} \mbox{in}\hspace{1.14mm} \mathbb{R}_{+}^{N+1},\\
\displaystyle  \partial_{\nu}^{s}w_{k}= -\lambda u_{k}+\big(|x|^{-\mu}\ast (H_{k}u_{k})\big)K_{k} +(\lambda-1)u  \hspace{7.6mm} \mbox{on}\hspace{1.14mm} \mathbb{R}^{N},
\end{array}
\right.
\end{equation}
where $w_{k}= E_{s}(u_{k})$ and $\partial_{\nu}^{s}w_{k}(x,0)\doteq-\frac{1}{\kappa_{s}}\lim_{y\rightarrow0^{+}}y^{1-2s}\frac{\partial w}{\partial y}(x,y)$, $\forall x\in\mathbb{R}^N$. Clearly, if $w_{k}$ is a weak solution of \eqref{EQP}, then $u_{k}=w_{k}(\cdot,0)$ is a weak solution of \eqref{ELP}. A weak solution to \eqref{EQP} is a function $w_{k}\in X^{s}(\mathbb{R}_{+}^{N+1})$ such that
\begin{equation} \label{EQP2}
\kappa_{s}\int_{\mathbb{R}_{+}^{N+1}}y^{1-2s}\nabla w_{k}\nabla\varphi =\int_{\mathbb{R}^N}\Big(-\lambda u_{k}+(|x|^{-\mu}\ast (H_{k}u_{k}))K_{k} +(\lambda-1)u\Big)\varphi ,
\end{equation}
for every $\varphi\in X^{s}(\mathbb{R}_{+}^{N+1}).$

For $T>0$, we denote the truncated function $w_{k,T}:\mathbb{R}_{+}^{N+1}\to \R$ by
$$
w_{k,T}=\left\{\begin{array}{l}
\displaystyle T \ \ \ \mbox{if}\ \ w_{k,T}\leq -T,\\
\displaystyle w_{k} \ \ \ \mbox{if} \ \  -T<w_{k,T}<T,\\
\displaystyle T \ \ \ \mbox{if}\ \ w_{k,T}\geq T\\
\end{array}
\right.
$$
and
$$
 u_{k,T}=w_{k,T}(\cdot,0).
$$
For $p\geq 2$, since $|w_{k,T}|^{p-2}w_{k,T}\in X^{s}(\mathbb{R}_{+}^{N+1})$, we have $|u_{k,T}|^{p-2}u_{k,T}\in H^{s}(\mathbb{R}^N)$.  Take $|u_{k,T}|^{p-2}u_{k,T}$ as a test function in \eqref{EQP2}, we obtain
$$\aligned
\frac{4(p-1)}{p^{2}}\int_{\mathbb{R}^N}|(-\Delta)^{\frac{s}{2}}&(u_{k,T})^{\frac{p}{2}}|^{2}+
\lambda ||u_{k,T}|^{\frac{p}{2}}|^{2}\\
&\leq \frac{4(p-1)}{p^{2}}\kappa_{s}\int_{\mathbb{R}_{+}^{N+1}}y^{1-2s}|\nabla (w_{k,T})^{\frac{p}{2}}|^{2}+\lambda \int_{\mathbb{R}^N}|u_{k,T}|^{p-2}u_{k,T}u_{k}\\
&= (p-1)\kappa_{s}\int_{\mathbb{R}_{+}^{N+1}}y^{1-2s}|w_{k,T}|^{p-2}|\nabla w_{k,T}|^{2}+\lambda \int_{\mathbb{R}^N}|u_{k,T}|^{p-2}u_{k,T}u_{k}\\
&\leq(p-1)\kappa_{s}\int_{\mathbb{R}_{+}^{N+1}}y^{1-2s}|w_{k}|^{p-2}\langle\nabla w_{k},\nabla w_{k,T}\rangle+\lambda \int_{\mathbb{R}^N}|u_{k,T}|^{p-2}u_{k,T}u_{k}\\
&=\int_{\mathbb{R}^N}(|x|^{-\mu}\ast (H_{k}u_{k}))(K_{k}|u_{k,T}|^{p-2}u_{k,T}) +(\lambda-1)u|u_{k,T}|^{p-2}u_{k,T}.
\endaligned
$$
If $p<\frac{2N}{N-\mu}$, using Lemma \ref{ELT} with $\theta=\frac{2}{p}$, there exists $C>0$ such that
$$\aligned
\int_{\mathbb{R}^N}(|x|^{-\mu}\ast (|H_{k}&u_{k,T}|))(|K_{k}||u_{k,T}|^{p-2}u_{k,T})\\
&\leq\int_{\mathbb{R}^N}(|x|^{-\mu}\ast (|H||u_{k,T}|))(|K||u_{k,T}|^{p-1})\\
&\leq\frac{2(p-1)}{p^{2}}\int_{\mathbb{R}^N}|(-\Delta)^{\frac{s}{2}}(u_{k,T})^{\frac{p}{2}}|^{2}+C\int_{\mathbb{R}^N}||u_{k,T}|^{\frac{p}{2}}|^{2}.
\endaligned
$$
So, we have
$$
\frac{2(p-1)}{p^{2}}\int_{\mathbb{R}^N}|(-\Delta)^{\frac{s}{2}}(u_{k,T})^{\frac{p}{2}}|^{2}
\leq C_{1}\int_{\mathbb{R}^N}(|u_{k}|^{p}+|u|^{p})+\int_{A_{k,T}}(|x|^{-\mu}\ast (|K_{k}||u_{k}|^{p-1}))|H_{k}u_{k}|,
$$
where $ A_{k,T}=\{x\in\mathbb{R}^N:|u_{k}|>T\}.$
Since $p<\frac{2N}{N-\mu}$, applying Hardy-Littlewood-Sobolev inequality again,
$$
\int_{A_{k,T}}(|x|^{-\mu}\ast (|K_{k}||u_{k}|^{p-1}))|H_{k}u_{k}|\leq C(\int_{\mathbb{R}^N}||K_{k}||u_{k}|^{p-1}|^{r})^{\frac{1}{r}}(\int_{A_{k,T}}|H_{k}u_{k}|^{l})^{\frac{1}{l}},
$$
with $\frac{1}{r}=1+\frac{N-\mu}{2N}-\frac{1}{p}$ and $\frac{1}{l}=\frac{N-\mu}{2N}+\frac{1}{p}$. By H\"{o}lder's inequality, if $u_{k}\in L^{p}(\mathbb{R}^N)$, then $|H_{k}||u_{k}|\in L^{l}(\mathbb{R}^N)$ and $|K_{k}||u_{k}|^{p-1}\in L^{r}(\mathbb{R}^N)$ , whence by Lebesgue's dominated convergence theorem
$$
\lim_{T\rightarrow\infty}\int_{A_{k,T}}(|x|^{-\mu}\ast (|K_{k}||u_{k}|^{p-1}))|H_{k}u_{k}|=0.
$$
By Lemma \ref{EMB}, we know that there exists $C_{2}>0$ such that
$$
\lim_{k\rightarrow\infty}\sup(\int_{\mathbb{R}^N}|u_{k}|^{\frac{pN}{N-2s}})^{1-\frac{2s}{N}}\leq C_{2}\lim_{k\rightarrow\infty}\sup\int_{\mathbb{R}^N}|u_{k}|^{p},
$$
and thus
$$
(\int_{\mathbb{R}^N}|u|^{\frac{pN}{N-2s}})^{1-\frac{2s}{N}}\leq C_3\big(\int_{\mathbb{R}^N}|u|^{2}\big)^{\frac12}.
$$
By iterating over $p$ a finite number of times we cover the range $p\in[2,\frac{2N}{N-\mu})$. So we can get weak solution $u\in L^{p}(\mathbb{R}^N)$ of \eqref{CFL} for every $p\in[2,\frac{2N^{2}}{(N-\mu)(N-2s)})$.
\end{proof}

\begin{lem}\label{INT22}
 Assume that $N\geq3$, $s\in(0,1)$ and $\mu\in(0,2s)$. If $f\in C(\mathbb{R},\mathbb{R})$ satisfies $(f_1)$, $(f_2)$ and $(f_3)$ and $u\in H^{s}(\mathbb{R}^N)$ solves equation \eqref{CFL}, then $u\in L^p(\R^N)$ for any $p\in [2, +\infty]$.
\end{lem}
\begin{proof}
We denote $H(x)=F(u(x))/u(x)$ and $K(x)=f(u(x)).$ By $(f_1)$, we know, for every $x\in\mathbb{R}^N$,
$$
H(x)\leq C(\frac12|u(x)|+\frac{N-2s}{2N-\mu}|u(x)|^{\frac{N+2s-\mu}{N-2s}})
$$
and
$$
K(x)\leq C(|u(x)|+|u(x)|^{\frac{N+2s-\mu}{N-2s}}).
$$
Thus $H,K\in L^{2}(\mathbb{R}^N)+L^{\frac{2N}{N+2s-\mu}}(\mathbb{R}^N)$. Applying \ref{Integrability}, we know the weak solution $u\in L^{p}(\mathbb{R}^N)$ for every $p\in[2,\frac{2N^{2}}{(N-\mu)(N-2s)})$. Using $(f_1)$, we know $F(u)\in L^{q}(\mathbb{R}^N)$ for every $q\in[\frac{2N}{2N-\mu},\frac{2N^{2}}{(N-\mu)(2N-\mu)})$. Since $\frac{2N}{2N-\mu}<\frac{N}{N-\mu}<\frac{2N^{2}}{(N-\mu)(2N-\mu)}$, we have $M(x):=|x|^{-\mu}\ast F(u)\in L^{\infty}(\mathbb{R}^N)$.

We will show that $u\in L^p(\R^N)$ for any $p\in [2, +\infty]$. Using the Dirichlet-to-Neumann map expression, we can see that the problem \eqref{CFL} can be transformed into the following local problem
\begin{equation}\label{CFLL2}
\left\{\begin{array}{l}
\displaystyle -div(y^{1-2s}\nabla w)=0\ \  \mbox{in}\ \ \mathbb{R}_{+}^{N+1},\\
\displaystyle  \partial_{\nu}^{s}w= -u+M(x)f(u)  \ \  \mbox{on}\ \  \mathbb{R}^{N},
\end{array}
\right.
\end{equation}
where
$$
\partial_{\nu}^{s}w(x,0)\doteq-\frac{1}{\kappa_{s}}\lim_{y\rightarrow0^{+}}y^{1-2s}\frac{\partial w}{\partial y}(x,y), \ \ \forall\ \  x\in\mathbb{R}^N.
$$

Since $w\in X^{s}(\mathbb{R}_{+}^{N+1})$ is a critical point such that
\begin{equation}\label{EE}
\kappa_{s}\int_{\mathbb{R}_{+}^{N+1}}y^{1-2s}\nabla w\nabla\varphi =\int_{\mathbb{R}^N}\big( -u+M(x)f(u)\big)\varphi ,
\end{equation}
for every $\varphi\in X^{s}(\mathbb{R}_{+}^{N+1}).$ For $T>0$, we denote
$$
w_{T}=\min\{w_+,T\}\ \ \mbox{and}\ \ u_{T}=w_{T}(\cdot,0),
$$
where $w_{+}=\max\{0,w\}$.  Since for $\beta>0$, $|w_{T}|^{2\beta}w\in X^{s}(\mathbb{R}_{+}^{N+1})$, take it as a test function in \eqref{EE}, we deduce that
$$\aligned
&\kappa_{s}\int_{\mathbb{R}_{+}^{N+1}}y^{1-2s}\langle\nabla w,\nabla (|w_{T}|^{2\beta}w)\rangle\\
&= \kappa_{s}\int_{\mathbb{R}_{+}^{N+1}}y^{1-2s}|w_{T}|^{2\beta}|\nabla w|^{2}+(2\beta)\kappa_{s}\int_{\{w\leq T\}}y^{1-2s}w^{2\beta}|\nabla w|^{2}\\
&=\int_{\mathbb{R}^N}\big(-u+M(x)f(u)\big)|u_{T}|^{2\beta}u .
\endaligned
$$
Notice that,
$$
\aligned
&\int_{\mathbb{R}_{+}^{N+1}}\kappa_{s}y^{1-2s}|\nabla (|w_{T}|^{\beta}w)|^{2}\\
&= \int_{\mathbb{R}_{+}^{N+1}}\kappa_{s}y^{1-2s}|w_{T}|^{2\beta}|\nabla w|^{2}+\big(2\beta+\beta^{2}\big)\int_{\{w\leq T\}}\kappa_{s}y^{1-2s}w^{2\beta}|\nabla w|^{2},
\endaligned
$$
we obtain,
$$
\aligned
\int_{\mathbb{R}^N}|(-\Delta)^{\frac{s}{2}}&(|u_{T}|^{\beta}u)|^{2}+ ||u_{T}|^{\beta}u|^{2}\\
&=\int_{\mathbb{R}_{+}^{N+1}}\kappa_{s}y^{1-2s}|\nabla (|w_{T}|^{\beta}w)|^{2}+\int_{\mathbb{R}^N} ||u_{T}|^{\beta}u|^{2}\\
&= \int_{\mathbb{R}_{+}^{N+1}}\kappa_{s}y^{1-2s}|w_{T}|^{2\beta}|\nabla w|^{2}+\int_{\mathbb{R}^N} ||u_{T}|^{\beta}u|^{2}\\
&\hspace{1cm}+\big(2\beta+\beta^{2}\big)\int_{\{w\leq T\}}\kappa_{s}y^{1-2s}w^{2\beta}|\nabla w|^{2}\\
&\leq C_{\beta}\int_{\mathbb{R}^N}M(x)f(u)|u_{T}|^{2\beta}u,
\endaligned
$$
where $C_\beta=1+\frac{\beta}{2}$
Since $M(x)\in L^{\infty}(\mathbb{R}^N)$, from assumption $(f_1)$, $$\frac{f(u)}{u}\leq C_1+|u|^{\frac{4s-\mu}{N-2s}}.$$
Since $u\in L^{p}(\mathbb{R}^N)$ for every $p\in[2,\frac{2N^{2}}{(N-\mu)(N-2s)})$, we know that, for
some constant $C_{1}$ and function $g\in L^{\frac{N}{2s}}(\mathbb{R}^N)$, $g\geq0$ and independent of $T$ and $p$,
\begin{equation}\label{MI1}
M(x)f(u)|u_{T}|^{2\beta}u \leq(C_{1}+g)|u_{T}|^{2\beta}u^{2}.
\end{equation}
So we have that
$$\aligned
\int_{\mathbb{R}^N}|(-\Delta)^{\frac{s}{2}}&(|u_{T}|^{\beta}u)|^{2}+ ||u_{T}|^{\beta}u|^{2}\\
&\leq C_{1}C_{\beta}\int_{\mathbb{R}^N}|u_{T}|^{2\beta}u^{2} +C_{\beta}\int_{\mathbb{R}^N}g|u_{T}|^{2\beta}u^{2} .
\endaligned
$$
and, using Fatou's lemma and monotone convergence, we can pass to the limit as
$T\rightarrow\infty$ to get
$$\aligned
\int_{\mathbb{R}^N}|(-\Delta)^{\frac{s}{2}}&(u_{+})^{\beta+1}|^{2}+ |(u_{+})^{\beta+1}|^{2}
\leq C_{1}C_{\beta}\int_{\mathbb{R}^N}|u_{+}|^{2(\beta+1)} +C_{\beta}\int_{\mathbb{R}^N}g|u_{+}|^{2(\beta+1)} .
\endaligned
$$
For any $M>0$, let $A_1=\{g\leq M\}$, $A_2=\{g> M\}$. Since
$$
\aligned
\int_{\mathbb{R}^N}g|u_{+}|^{2(\beta+1)}&=\int_{A_1}g|u_{+}|^{2(\beta+1)}+\int_{A_2}g|u_{+}|^{2(\beta+1)}\\
&\leq M \int_{A_1}|u_{+}|^{2(\beta+1)}+\big(\int_{A_2}g^{\frac{N}{2s}}\big)^{\frac{2s}{N}}\big(\int_{A_2}|u_{+}|^{(\beta+1)\frac{2N}{N-2s}}\big)^{\frac{N-2s}{N}}\\
&\leq M|{u_{+}}^{(\beta+1)}|^2_2+\vr(M)|{u_{+}}^{(\beta+1)}|^2_{2^*(s)},
\endaligned
$$
we can get,
$$
\|{u_{+}}^{(\beta+1)}\|^2_H\leq C_{\beta}(C_1+M)|{u_{+}}^{(\beta+1)}|^2_2+C_{\beta}\vr(M)|{u_{+}}^{(\beta+1)}|^2_{2^*(s)}.
$$
Using Lemma \ref{EMB} and taking $M$ large enough such that $C_{\beta}C_{2^*(s)}^2\vr(M)\leq \frac12$, we obtain
$$\aligned
|{u_{+}}^{(\beta+1)}|^2_{2^*(s)}\leq C_{\beta}C_{2^*(s)}^2(C_1+M)|{u_{+}}^{(\beta+1)}|^2_2.
\endaligned
$$
Now a bootstrap argument start with $\beta+1=\frac{N}{N-2s}$ show that $u_{+} \in L^p(\R^N)$ for any $p\in [2, +\infty)$. Similarly, we can see $u_{-}\in L^p(\R^N)$ for any $p\in [2, +\infty)$ and hence it holds for $u$.

Now we are ready to show that $u$ is in fact bounded in $\R^N$.
Since show that $u_{+} \in L^p(\R^N)$ for any $p\in [2, +\infty)$, repeat the arguments in \eqref{MI1}, we know there exists some constant $C_{1}$ and function $g\in L^{\frac{N}{s}}(\mathbb{R}^N)$, $g\geq0$ and independent of $T$ and $\beta$ such that
$$\aligned
\int_{\mathbb{R}^N}|(-\Delta)^{\frac{s}{2}}&(|u_{T}|^{\beta}u)|^{2}+ ||u_{T}|^{\beta}u|^{2}\leq C_{1}C_{\beta}\int_{\mathbb{R}^N}|u_{T}|^{2\beta}u^{2} +C_{\beta}\int_{\mathbb{R}^N}g|u_{T}|^{2\beta}u^{2} .
\endaligned
$$
Using Fatou's lemma and monotone convergence, we can pass to the limit as
$T\rightarrow\infty$ to get
$$\aligned
\int_{\mathbb{R}^N}|(-\Delta)^{\frac{s}{2}}&(u_{+})^{\beta+1}|^{2}+ |(u_{+})^{\beta+1}|^{2}
\leq C_{1}C_{\beta}\int_{\mathbb{R}^N}|u_{+}|^{2(\beta+1)} +C_{\beta}\int_{\mathbb{R}^N}g|u_{+}|^{2(\beta+1)} .
\endaligned
$$
Using Young's inequality, we see
$$
\int_{\mathbb{R}^N}g|u_{+}|^{2(\beta+1)} \leq|g|_{\frac{N}{s}}|(u_{+})^{\beta+1}|_{2}|(u_{+})^{\frac{p}{2}}|_{2^{\ast}(s)}\leq
|g|_{\frac{N}{s}}(\lambda|(u_{+})^{\beta+1}|_{2}^{2}+\frac{1}{\lambda}|(u_{+})^{\beta+1}|_{2^{\ast}(s)}^{2}).
$$
Therefore,
$$
\int_{\mathbb{R}^N}|(-\Delta)^{\frac{s}{2}}(u_{+})^{\beta+1}|^{2}+ |(u_{+})^{\beta+1}|^{2}\leq C_{\beta}(C_{1}+|g|_{\frac{N}{s}}\lambda)|(u_{+})^{\beta+1}|_{2}^{2}+\frac{C_{\beta}|g|_{\frac{N}{s}}}{\lambda}|(u_{+})^{\beta+1}|_{2^{\ast}(s)}^{2}.
$$
Using Lemma \ref{EMB} and taking $\lambda$ large enough such that $\frac{C_{\beta}|g|_{\frac{N}{s}}}{\lambda}C_{2^*(s)}^{2}=\frac{1}{2}$,
we obtain
$$\aligned
|(u_{+})^{\beta+1}|_{2^{\ast}(s)}^{2}\leq 2C_{2^*(s)}^{2}C_{\beta}(C_{1}+|g|_{\frac{N}{s}}\lambda)|(u_{+})^{\beta+1}|_{2}^{2}=M_{\beta}|(u_{+})^{\beta+1}|_{2}^{2}.
\endaligned
$$
Since $M_{\beta}\leq CC^2_{\beta}\leq C(1+\beta)^2\leq M_0e^{2\sqrt{1+\beta}}$, we know
$$\aligned
|u_{+}|_{2^{\ast}(s)(\beta+1)}\leq M_{0}^{1/(\beta+1)}e^{1/\sqrt{1+\beta}}|u_{+}|_{2(\beta+1)}.
\endaligned
$$
Start with  $\beta_{0}=0$, $2(\beta_{n+1}+1)=2^{\ast}(s)(\beta_{n}+1)$, an iteration shows
$$\aligned
|u_{+}|_{2^{\ast}(s)(\beta_{n}+1)}\leq M_{0}^{\sum_{i=0}^{n}1/(\beta_{i}+1)}e^{\sum_{i=0}^{n}1/\sqrt{\beta_{i}+1}}|u_{+}|_{2(\beta_{0}+1)}.
\endaligned
$$
Since $\beta_{n}+1=(\frac{2^{\ast}(s)}{2})^{n}=(\frac{N}{N-2s})^{n}$, we can get that,
$$
\sum_{i=0}^{\infty}1/(\beta_{i}+1)<\infty\hspace{2.14mm}\mbox{and}\hspace{2.14mm}\sum_{i=0}^{\infty}1/\sqrt{\beta_{i}+1}<\infty
$$
and from this we deduce that
$$
|u_{+}|_{\infty}=\lim_{n\rightarrow\infty}|u_{+}|_{2^{\ast}(s)(\beta_{n}+1)}<\infty.
$$
Thus, $u_{+}\in L^{\infty}(\mathbb{R}^N)$. Clearly, the same is true for $u_{-}$ and hence for $u$.
\end{proof}

In fact we can show that the solutions for equation \eqref{CFL} possess some regularity if the nonlinearity $f$ has some more regularity.
\begin{Prop}\label{Regu1} [Proposition 2.8, \cite{Si}]
Let $w=(-\Delta)^{s}u$. Assume $w\in C^{0,\alpha}(\mathbb{R}^{n})$ and $u\in L^{\infty}(\mathbb{R}^{n})$, for $\alpha\in(0,1]$ and $\sigma>0$.

$(i).$ If $\alpha+2s\leq1$, then $u\in C^{0,\alpha+2s}(\mathbb{R}^{n})$. Moreover,
$$
\|u\|_{C^{0,\alpha+2s}(\mathbb{R}^{n})}\leq C(\|u\|_{L^{\infty}}+\|w\|_{C^{0,\alpha}})
$$

for a constant $C$ depending only on $n$, $\alpha$ and $s$.

$(ii).$ If $\alpha+2s>1$, then $u\in C^{1,\alpha+2s-1}(\mathbb{R}^{n})$. Moreover,
$$
\|u\|_{C^{1,\alpha+2s-1}(\mathbb{R}^{n})}\leq C(\|u\|_{L^{\infty}}+\|w\|_{C^{0,\alpha}})
$$
\end{Prop}
for a constant $C$ depending only on $n$, $\alpha$ and $s$.

\begin{Prop}\label{Regu2}  [Proposition 2.9, \cite{Si}]
Let $w=(-\Delta)^{s}u$. Assume $w\in L^{\infty}(\mathbb{R}^{n})$ and $u\in L^{\infty}(\mathbb{R}^{n})$ for $s>0$.

$(i).$ If $2s\leq1$, then $u\in C^{0,\alpha}(\mathbb{R}^{n})$ for any $\alpha<2s$. Moreover,
$$
\|u\|_{C^{0,\alpha}(\mathbb{R}^{n})}\leq C(\|u\|_{L^{\infty}}+\|w\|_{L^{\infty}})
$$

for a constant $C$ depending only on $n$, $\alpha$ and $s$.

$(ii).$ If $2s>1$, then $u\in C^{1,\alpha}(\mathbb{R}^{n})$ for any $\alpha<2s-1$. Moreover,
$$
\|u\|_{C^{1,\alpha}(\mathbb{R}^{n})}\leq C(\|u\|_{L^{\infty}}+\|w\|_{L^{\infty}})
$$

for a constant $C$ depending only on $n$, $\alpha$ and $s$.
\end{Prop}
Using the Lemma \ref{INT22} and Propositions above, we know that the weak solution $u$ is in fact a classical solution of \eqref{CFL}. From Lemma \ref{INT22}, we know that $u\in L^{\infty}(\mathbb{R}^N)$.  From Proposition \ref{Regu2}, we get some $\sigma\in(0,1)$ depending on $s$ such
that $u\in C^{0,\sigma}(\mathbb{R}^N)$ and hence $[(|x|^{-\mu}\ast F(u))f(u)-u]\in C^{0,\alpha}(\mathbb{R}^N)$ for some $\alpha\in(0,1)$. Thus $u\in C^{0,\alpha+2s}(\mathbb{R}^N)$. If $\alpha+2s>1$, following the proof of Proposition \ref{Regu1}, we can obtain that $u\in C^{1,\alpha+2s-1}(\mathbb{R}^N)$. What's more
\begin{Prop}\label{PropClassical} Assume that $f\in C^1(\R,\R)$, then the weak solution $u\in C^{2,\alpha}(\mathbb{R}^N)$ for some $\alpha$ depending on $s$ and satisfies
$$
\frac{N-2s}{2}\int_{\mathbb{R}^N}|(-\Delta)^{\frac{s}{2}}u|^{2}+\frac{N}{2}\int_{\mathbb{R}^N} |u|^{2}=\frac{2N-\mu}{2}\int_{\mathbb{R}^N}(|x|^{-\mu}\ast F(u))F(u).
$$
\end{Prop}
\begin{proof}

Let $s\in(1/2,1)$ and $u$ in $ L^{\infty}(\mathbb{R}^{N})$ be a solution of the equation
$$
(-\Delta)^{s}u+ u
=\big(|x|^{-\mu}\ast F(u)\big)f(u)\hspace{4.14mm}\mbox{in}\hspace{1.14mm} \mathbb{R}^N.
$$
Denote by $W(u(x))=\big(|x|^{-\mu}\ast F(u)\big)f(u)(x)$, we know $W\in C^1$.  Applying $(ii)$ of Proposition \ref{Regu2}, we know
$u\in C^{1,\alpha}(\mathbb{R}^{N})$. Moreover $\partial_{x_i}u$ satisfies
$$
(-\Delta)^{s}\partial_{x_i}u(x)
=\partial_{x_i}\big( W(u(x))\big),\hspace{4.14mm}\mbox{for any}\hspace{1.14mm} x\in\mathbb{R}^N.
$$
Applying $(ii)$ of Proposition \ref{Regu2} to $\partial_{x_i}u(x)$ again, it follows that $\partial_{x_i}u$ belongs to $ C^{1,\alpha}(\mathbb{R}^{N})$ for any $\alpha<2s-1$ and thus the claim is proved.

Let $s=1/2$. Since $W$ is in $C^{1}$, $(i)$ of proposition \ref{Regu2} implies $u\in C^{0,\alpha}(\mathbb{R}^{N})$ for any $\alpha<1$. Then $(ii)$ of proposition \ref{Regu1}  with $w:=W(u)$ yields that the function $u$ belongs to $ C^{1,\alpha}(\mathbb{R}^{N})$ for any $\alpha<1$. Now, we can argue as for the case $s\in(1/2,1)$ to obtain the desired
regularity for $u$.

Finally, let $s\in(0,1/2)$ and let $u\in L^{\infty}(\mathbb{R}^{N})$ be the solution. Then $(i)$ of proposition \ref{Regu2} implies $u\in C^{0,\alpha}(\mathbb{R}^{N})$ for any $\alpha<2s$. Then, for $s\in(1/4,1/2)$ we can apply $(ii)$ of proposition \ref{Regu1} and we get $u\in C^{1,\alpha+2s-1}(\mathbb{R}^{N})$. Hence, $\partial_{x_i}u$ is well defined with $\partial_{x_i}\big( W(u(x))\big)$ belonging to $C^{0,\alpha+2s-1}(\mathbb{R}^{N})$ and again by $(ii)$ of proposition \ref{Regu1} we get $\partial_{x_i}u\in C^{1,\alpha+2s-1}(\mathbb{R}^{N})$ for any $\alpha<2s$.

For $s\in(0,1/4]$, by $(ii)$ of proposition \ref{Regu1},  we know $u\in C^{0,\alpha+2s}(\mathbb{R}^{N})$ for any
$\alpha<2s$. Thus, when $s\in(1/6, 1/4]$, apply $(ii)$ of proposition \ref{Regu1} twice  and argue as in the
case $s\in(1/4,1/2)$ and we get $\partial_{x_i}u\in C^{1,\alpha+4s-1}(\mathbb{R}^{N})$, for any $\alpha<2s$.

By iterating the above procedure on $k\in\mathbb{N}$, we obtain that, when $s\in(1/(2k+2),1/2k]$,
$u$ belongs to $u\in C^{2,\alpha+2k-1}(\mathbb{R}^{N})$ for any $\alpha<2s$.

Since $u\in C^{2,\alpha}(\mathbb{R}^N)$, we have $w\in C^{2,\alpha}(\mathbb{R_{+}}^{N+1})$, where $w=E_{s}(u)$. We denote $\mathbb{D}=\{z=(x,y)\in\mathbb{R}^N\times[0,\infty):|z|\leq1\}$. Fix $\varphi\in C_{c}^{1}(\mathbb{R_{+}}^{N+1})$ such that $\varphi=1$ on $\mathbb{D}$ and $\varphi_{\lambda }:=\varphi(\lambda x,\lambda y)$. The function $w_{\lambda}$ defined for $\lambda\in(0,\infty)$ and $z\in\mathbb{R_{+}}^{N+1}$ by
$$
w_{\lambda}(z)=\varphi_{\lambda }z\cdot\nabla w(z)
$$
can be used as a test function in the equation to obtain
$$
\kappa_{s}\int_{\mathbb{R}_{+}^{N+1}}y^{1-2s}\nabla w\nabla w_{\lambda}dz =\int_{\mathbb{R}^N}\big( -u+\big(|x|^{-\mu}\ast F(u)\big)f(u)\big)w_{\lambda}(x,0) dx.
$$
From the arguments in [15, Theorem 6.1] and [28, Proposition 3.5], we know
$$
\lim_{\lambda\rightarrow 0}\int_{\mathbb{R}_{+}^{N+1}}y^{1-2s}\nabla w\nabla w_{\lambda}dz=-\frac{N-2s}{2}\int_{\mathbb{R}_{+}^{N+1}}y^{1-2s}|\nabla w|^{2}dz,
$$
$$
\lim_{\lambda\rightarrow 0}\int_{\mathbb{R}^N}uw_{\lambda}(x,0)=-\frac{N}{2}\int_{\mathbb{R}^N} |u|^{2}
$$
and
$$
\lim_{\lambda\rightarrow 0}\int_{\mathbb{R}^N}\big(|x|^{-\mu}\ast F(u)\big)f(u)w_{\lambda}(x,0) =-\frac{2N-\mu}{2}\int_{\mathbb{R}^N}(|x|^{-\mu}\ast F(u))F(u).
$$
The conclusion follows the above equalities and \eqref{Norm4}.

\end{proof}

\section{Proof of the main results}

\textbf{Proof of Theorem $1.3$.}
Since $u_{0}$ is a nontrivial solution of \eqref{CFL}, we have $J(u_{0})\geq c$. From Lemma \ref{Existence} and by the definition of the ground state energy level $c$, we can get $c\leq c^{\star}$.

Follow the idea of Jeanjean and Tanaka [20, lemma 2.1], we define the path $\tilde{\gamma}:[0,\infty)\rightarrow H^{s}(\mathbb{R}^N)$ by
$$
\tilde{\gamma}(\tau)(x)=\left\{\begin{array}{l}
\displaystyle u_{0}(x/\tau) \hspace{6.14mm} \mbox{if}\hspace{1.14mm} \tau>0,\\
\displaystyle 0 \hspace{16.6mm} \mbox{if} \hspace{1.14mm}\tau=0.\\
\end{array}
\right.
$$
Since the function $\tilde{\gamma}$ is continuous on $(0,\infty)$ and (2.1), we have, for every $\tau>0$,
$$
\int_{\mathbb{R}^N}|(-\Delta)^{\frac{s}{2}}\tilde{\gamma}(\tau)|^{2}+\int_{\mathbb{R}^N} |\tilde{\gamma}(\tau)|^{2}=\tau^{N-2s}\int_{\mathbb{R}^N}|(-\Delta)^{\frac{s}{2}}u_{0}|^{2}+\tau^{N}\int_{\mathbb{R}^N} |u_{0}|^{2},
$$
which implies $\tilde{\gamma}$ is continuous at 0. By the Poho\v{z}aev identity in Proposition \ref{PropClassical} and (2.1), the functional $J(\tilde{\gamma}(\tau))$ can be computed for every $\tau>0$ as
$$\aligned
J(\tilde{\gamma}(\tau))&=\frac{\tau^{N-2s}}{2}\int_{\mathbb{R}^N}|(-\Delta)^{\frac{s}{2}}u_{0}|^{2}+\frac{\tau^{N}}{2}\int_{\mathbb{R}^N} |u_{0}|^{2}-\frac{\tau^{2N-\mu}}{2}\int_{\mathbb{R}^N}(|x|^{-\mu}\ast F(u_{0}))F(u_{0})\\
&=(\frac{\tau^{N-2s}}{2}-\frac{(N-2s)\tau^{2N-\mu}}{2(2N-\mu)})\int_{\mathbb{R}^N}|(-\Delta)^{\frac{s}{2}}u_{0}|^{2}
+(\frac{\tau^{N}}{2}-\frac{N\tau^{2N-\mu}}{2(2N-\mu)})\int_{\mathbb{R}^N} |u_{0}|^{2}.
\endaligned
$$
It is easy to see that $J(\tilde{\gamma}(\tau))$ achieves strict global maximum at 1: for every $\tau\in[0,1)\cup(1,\infty)$, $J(\tilde{\gamma}(\tau))<J(u_{0})$. Then after a suitable change of variable, for every $t_{0}\in(0,1)$, there exists a path $\gamma\in C([0,1];H^{s}(\mathbb{R}^N))$ such that
$$\aligned
\gamma&\in\Gamma,\\
\gamma(t_{0})&=u_{0},\\
J(\gamma(t))&<J(u_{0}), \hspace{5.14mm}\forall t\in[0,t_{0})\cup(t_{0},1].
\endaligned
$$
Let $v_{0}\in H^{s}(\mathbb{R}^N)\backslash\{{0}\} $ be another solution of \eqref{CFL} such that $J(v_{0})\leq J(u_{0})$. If we lift $v_{0}$ to a
path and recall the definition (1.4) of $ c^{\star}$, we conclude that $J(u_{0})\leq c^{\star}\leq J(v_{0})$. Then, we have proved that $J(u_{0})=J(v_{0})=c=c^{\star}$, and this concludes the proof of
Theorem \ref{EXS}.$ \hspace{13.14mm} \Box$

\begin{Rem}\label{R1}
 If $f\in C^1(\mathbb{R},\mathbb{R})$ satisfies $(f_1)$ is and odd function that $f\geq0$ on $(0,\infty)$, then any groundstate $u$ of \eqref{CFL} is positive. To observe this, we recall that there exists a path $\gamma\in C([0,1];H^{s}(\mathbb{R}^N))$ such that
$$\aligned
\gamma&\in\Gamma,\\
\gamma(t_{0})&=u,\\
J(\gamma(t))&<J(u), \hspace{5.14mm}\forall t\in[0,t_{0})\cup(t_{0},1].
\endaligned
$$
Notice that
$$
\int_{\mathbb{R}^{N}}|(-\Delta)^{\frac{s}{2}}|u||^{2}=\int_{\mathbb{R}^{N}}|(-\Delta)^{\frac{s}{2}}u|^{2},
$$
we have
$$
J(u)=J(|u|).
$$
Hence for any $t\in[0,t_{0})\cup(t_{0},1]$, there holds
$$
J(|\gamma(t)|)=J(\gamma(t))<J(u)=J(|u_{0}|).
$$
From Lemma 5.1 in \cite{MS2}, we know $|u|$ is also a groundstate. If $u(x_0)=0$ for some $x_0\in\R^N$, then one obtains $$
\int_{\mathbb{R}^N}\frac{|u(x_0+y)|+|u(x_0-y)|}{|x-y|^{N+2s}}dy=0,
$$
implying that $u=0$, a contradiction, thus $|u|>0$.
\end{Rem}

\begin{Rem}\label{R2}
Denote by
$$
G_{c}=\{u \in H^{s}(\mathbb{R}^N): J(u)=c \hspace{1.14mm}\hbox{and}\hspace{1.14mm} u\hspace{1.14mm} \hbox{is a groundstate of (1.1)} \},
$$
i.e. the set of groundstates of \eqref{CFL}, then $G_{c}$ is compact in $ H^{s}(\mathbb{R}^N)$ up to translations in $\mathbb{R}^N$.
In fact, for every $u \in G_{c}$, we have
$$
J(u)=\frac{N-\mu+2s}{2(2N-\mu)}\int_{\mathbb{R}^N}|(-\Delta)^{\frac{s}{2}}u|^{2}+\frac{N-\mu}{2(2N-\mu)}\int_{\mathbb{R}^N} |u|^{2}.
$$
Thus, for every $(u_{n})_{n\in \mathbb{N}} \in G_{c}$, up to a subsequence and translations, we can assume that
$u_{n}\rightharpoonup u$. Notice that
$$\aligned
J(u)&=\frac{N-\mu+2s}{2(2N-\mu)}\int_{\mathbb{R}^N}|(-\Delta)^{\frac{s}{2}}u|^{2}dx+\frac{N-\mu}{2(2N-\mu)}\int_{\mathbb{R}^N} |u|^{2}dx\\
&=\lim_{n\rightarrow\infty}\inf\frac{N-\mu+2s}{2(2N-\mu)}\int_{\mathbb{R}^N}|(-\Delta)^{\frac{s}{2}}u_{n}|^{2}dx+\frac{N-\mu}{2(2N-\mu)}\int_{\mathbb{R}^N} |u_{n}|^{2}dx,
\endaligned
$$
and hence $(u_{n})_{n\in \mathbb{N}}$ converges strongly to $u$ in $H^{s}(\mathbb{R}^N)$.
\end{Rem}

Involving the symmetric property of groundstates, we need to recall some elements of the theory of polarization \cite{BS, MS1, MS2}. Assume that $H\subset\mathbb{R}^N$ is a closed half-space and that $\sigma_{H}$ is the reflection with respect to $\partial H$. The
polarization $u^{H}:\mathbb{R}^N\rightarrow\mathbb{R}$ of $u^{H}:\mathbb{R}^N\rightarrow\mathbb{R}$ is defined for $x\in\mathbb{R}^N$ by
$$
u^{H}(x)=\left\{\begin{array}{l}
\displaystyle \max(u(x),u(\sigma_{H}(x))) \hspace{9.6mm} \mbox{if}\hspace{1.14mm} x\in H,\\
\displaystyle \min(u(x),u(\sigma_{H}(x))) \hspace{10.6mm} \mbox{if} \hspace{1.14mm}x\not\in H.\\
\end{array}
\right.
$$
We denote $w^{H}=E_{s}(u^{H})$, where $E_{s}(u^{H})$ is $s-harmonic$
extension of $u^{H}$.

\begin{lem}\label{Polar} If $u \in H^{s}(\mathbb{R}^N)$, then $u^{H} \in H^{s}(\mathbb{R}^N)$ and
$$
\int_{\mathbb{R}^{N}}|(-\Delta)^{\frac{s}{2}}u^{H}|^{2}dx=\int_{\mathbb{R}^{N}}|(-\Delta)^{\frac{s}{2}}u|^{2}dx.
$$
\end{lem}
\begin{proof}
From Lemma 5.4 in \cite{MS2}, there holds
$$
\int_{\mathbb{R}^{N}}|\nabla u^{H}|^{2}dx=\int_{\mathbb{R}^{N}}|\nabla u|^{2}dx.
$$
So, for every $y\in \mathbb{R}_{+}$,
$$
\int_{\mathbb{R}^{N}}|\nabla w^{H}|^{2}dx
=\int_{\mathbb{R}^{N}}|\nabla w|^{2}dx.
$$
Applying \eqref{Norm4}, we have
$$
\int_{\mathbb{R}^{N}}|(-\Delta)^{\frac{s}{2}}u^{H}|^{2}dx=\kappa_{s}\int_{\mathbb{R}_{+}^{N+1}}y^{1-2s}|\nabla w^{H}|^{2}dxdy
$$
$$
\hspace{23.14mm}=\kappa_{s}\int_{\mathbb{R}_{+}^{N+1}}y^{1-2s}|\nabla w|^{2}dxdy=\int_{\mathbb{R}^{N}}|(-\Delta)^{\frac{s}{2}}u|^{2}dx.
$$
\end{proof}

\begin{lem}\label{Polar2} ([27, lemma 5.4]). Assume that $u\in L^{2}(\mathbb{R}^{N})$ is nonnegative. There exist $x_{0}\in \mathbb{R}^{N}$ and a nonincreasing function $v:(0,\infty)\rightarrow\mathbb{R}$ such that for almost every $x\in \mathbb{R}^{N}$, $u(x)=v(|x-x_{0}|)$ if and only if either $u^{H}=u$ or $u^{H}=u(\sigma_{H})$.
\end{lem}
\begin{lem} \label{R3}
  If $f\in C^1(\mathbb{R},\mathbb{R})$ satisfies $(f_1)$ is and odd function that $f\geq0$ on $(0,\infty)$, then any groundstate $u \in H^{s}(\mathbb{R}^N)$ of \eqref{CFL} is radially symmetric about a point.
\end{lem}
\begin{proof}
Similar to the arguments in Remark \ref{R1}, there exists $\gamma(t_0)=u$ and for every $t\in[0,1]$, $\gamma(t)\geq0.$ For every $H$, we define the path $\gamma^{H}(t)=(\gamma(t))^{H}\in[0,1]\rightarrow H^{s}(\mathbb{R}^N)$. By lemma \ref{Polar}, we know $\gamma^{H}$ is continuous. Since $F$ is nondecreasing, $F(u^{H})=(F(u))^{H}$, lemma \ref{Polar} and lemma 5.5 in \cite{MS2} imply that
$$
J(\gamma^{H}(t))\leq J(\gamma(t)).
$$
Thus $\gamma^{H}\in\Gamma$ and so $\max_{t\in[0,1]}J(\gamma^{H}(t))\geq c^{\star}$.
Since for every $t\in[0,t_0)\cup(t_0,1]$,
$$
J(\gamma^{H}(t))<J(\gamma(t))<c^{\star},
$$
from Lemma 5.1 in \cite{MS2}, we get
$$
J(\gamma^{H}(t_0))=J(\gamma(t_0))=c^{\star},
$$
Then, we have either $(F(u))^{H}=F(u)$ or $F(u^{H})=F(u(\sigma_{H}))$ in $\mathbb{R}^N$. Repeat the arguments in section 5 of \cite{MS2}, we know $u^{H}=u$ or $u^{H}=u(\sigma_{H})$.  Consequently lemma \ref{Polar2} implies $u(x)=v(|x-x_{0}|)$.\end{proof}

\end{document}